\documentclass[tm]{birkjour_t2}

\usepackage[colorlinks]{hyperref}
\usepackage[short,24hr]{datetime}

\usepackage{amsthm,array}
\usepackage{amsmath}

\usepackage{cite,color}
\usepackage[nameinlink]{cleveref}

\usepackage{tikz}
\usetikzlibrary{ intersections,calc,positioning,arrows, shapes,quotes,math}
\usepackage{pgfplots}
\pgfplotsset{compat=1.12}

\usepackage{enumerate}

\newtheorem{theorem}{\rm\bf Theorem}[section]
\newtheorem{proposition}[theorem]{\rm\bf Proposition}

\newtheorem{corollary}[theorem]{\rm\bf Corollary}

\newtheorem*{theorem*}{Theorem}
\newtheorem*{theorem 1}{\rm\bf Proposition 1}
\newtheorem*{theorem 2}{\rm\bf Proposition 2}

\theoremstyle{definition}
\newtheorem{definition}[theorem]{\rm\bf Definition}

\theoremstyle{remark}
\newtheorem{remark}[theorem]{\rm\bf Remark}

\newtheorem{example}[theorem]{\rm\bf Example}

\def\half#1#2{\begin{matrix}\frac{#1}{#2}\end{matrix}}

\def\R#1{\mathbb{R}^{#1}}

\def\scal#1#2{\langle #1; #2 \rangle}

\def\field{K}

\DeclareMathOperator{\Char}{char} 
\DeclareMathOperator{\Idm}{Idm}
\DeclareMathOperator{\Nil}{Nil_2}

\DeclareMathOperator{\trace}{tr}

\begin{document}

\title{Variety of idempotents in nonassociative algebras \\
  \phantom{A}
\\ \footnotesize{\color{red}{Ver.: \today, \currenttime}}
}

\author{Yakov Krasnov}
\address{Department of Mathematics, Bar-Ilan University, Ramat-Gan, 52900, Israel}
\email{krasnov@math.biu.ac.il}
\author{Vladimir G. Tkachev}
\address{Department of Mathematics, Link\"oping University, Link\"oping, 58183, Sweden}
\email{vladimir.tkatjev@liu.se}

\begin{abstract}
In this paper, we study the variety of all nonassociative (NA) algebras from the idempotent point of view. We are interested, in particular, in the spectral properties of idempotents when algebra is generic, i.e. idempotents are in  general position. Our main result states that in this case, there exist at least $n-1$ nontrivial obstructions (syzygies) on the Peirce spectrum of a generic NA algebra of dimension $n$. We also discuss the exceptionality of the eigenvalue $\lambda=\frac12$ which appears in the spectrum of idempotents in many classical examples of NA algebras and characterize its extremal properties in metrised algebras.
\end{abstract}

\vspace*{-1.5cm}
\maketitle




\section{Introduction}

The Peirce decomposition is a central tool of nonassociative algebra. In associative algebras (for example in matrix algebras), idempotents are projections onto subspaces, with eigenvalues 1 and 0 and play a distinguished role. In nonassociative algebras the spectrum of an idempotent (which is known also as the Peirce numbers) can be vary arbitrarily. Still, many classical examples of  nonassociative algebras share the following basic feature: the set of idempotents in algebra is rich enough (for example spans or generates the algebra) while the number of possible distinct Peirce numbers is few in comparison with the algebra dimension.

Let $A$ be a commutative nonassociative algebra over a filed $K$ of $\Char(K)=0$ unless otherwise stated explicitly.  Any semisimple idempotent $0\ne c=c^2\in A$ induces the corresponding  Peirce decomposition:
$$
A=\bigoplus_{\lambda\in\sigma(c)}A_c(\lambda),
$$
where $cx=xc=\lambda x$ for any $x\in A_c(\lambda)$ and $\sigma(c)$ is the Peirce spectrum of $c$. The Peirce spectrum $\sigma(A)=\{\lambda_1,\ldots, \lambda_s\}$ of the algebra $A$ is the set of all possible distinct eigenvalues $\lambda_i$ in $\sigma(c)$, when $c$ runs  all idempotents of $A$.
A fusion (or multiplication) rule is the inclusion of the following kind:
$$
A_c(\lambda_i)A_c(\lambda_j)\subset\bigoplus_{k\in\mathcal{F}(i,j)} A_{\lambda_k}, \quad \mathcal{F}(i,j)\subset \{1,2\,\ldots,s\}
$$

For example, if $A$ is power-associative (for instance, $A$ is a Jordan algebra) then it its Peirce spectrum (i.e. the only possible Peirce numbers of $A$) is $\sigma(A)=\{0,\half12,1\}$, see \cite{Albert48}, \cite{Schafer}. The middle value $\half12$ is crucial for structural properties and classification of formally real Jordan algebras: a Jordan algebra $A$ is simple if and only the corresponding eigenspace $A_c(\half12)$ is nontrivial for any nonzero idempotent $c\in A$ \cite[p.~63]{FKbook}. It is also well-known that the fusion rules of a Jordan algebra are $\mathbb{Z}/2$-graded for any idempotent $c\in A$: if $A^0=A_c(0)\oplus A_c(1)$ and $A^1=A_c(\half12)$ then $A^iA^j\subset A^{i+j \mod 2}$.

Another important example is axial algebras  appearing in connection with the Monster sporadic simple group \cite{Norton94}. The most famous example here is the Griess algebra generated by idempotents with Peirce numbers $1$, $0$, $\half14$, $\half1{32}$ and satisfying the so-called Ising fusion rules \cite{HRS15}, \cite{Rehren17}, \cite{Ivanov15}. These fusion rules are also $\mathbb{Z}/2$-graded.

We also mention  very recent examples of the so-called Hsiang algebras appearing in the classification of cubic minimal cones (the REC-algebras in terminology of \cite[Chapter~6]{NTVbook}). The Peirce spectrum of such an algebra consists of  four numbers: $\sigma(A)=\{1,-1,-\half12,\half12\}$. The Hsiang algebras  have a nice fusion rules but they are not graded. Furthermore, these algebras share a remarkable property: all idempotents have the same spectrum. As we shall see below, the latter property is closely related to the fact that $\half12$ belongs to the algebra spectrum.

Traditionally, one defines a (nonassociative) algebra structure by virtue of algebra identities (for example, Lie, Jordan and power-associative algebras) or a multiplication table (for example, division algebras or evolution algebras \cite{TianEvol}). Also, an algebraic structure can be defined by postulating some distinguished properties of idempotents, as for example for the axial algebras mentioned before.

\smallskip
One of the main goals of the present paper is to support the following conjectural paradigm: \emph{Many essential or invariant properties of a non-associative algebra can be recovered directly from its Peirce spectrum.} In other words, the knowledge of the Peirce spectrum of an algebra allows one to determine the most important features of the algebra.

\smallskip
In this connection, the following principal questions arise and will be discussed in this paper.

\begin{enumerate}[a)]
\item How arbitrary can the Peirce spectrum   be? What kind of obstructions (syzygyies) can exist?

\item What can be said about the structure of an algebra with a prescribed set of the Peirce numbers?

\item Which Peirce numbers have  `distinguished' properties?

\item
If the Peirce spectrum is known, what can be said about the possible multiplicities of the eigenvalues? For example, when the spectrum is independent of a choice of an idempotent?

\item
Do there exist any obstructions/syzigies on the fusion table?
\end{enumerate}

Thus formulated program is rather ambitious even for characteristic 0. To obtain some significant results, we sometimes assume that $\field$ is a subfield of $\mathbb{C}$. We emphasize that in this paper we are a more interested in discussing and illustrating some new methods and phenomena with a clear analytical or topological flavor. We outline only some of the possible directions and obtain some particular answers on the above questions.

Our main result here describes syzygies (obstructions) on the idempotent set of a finite dimensional generic commutative nonassociative algebra $A$ under fairly general assumptions. More precisely, we show (see \Cref{th4} below) that
$$
\sum_{c\in \Idm_0 (A)}\frac{\chi_c(t)}{\chi_c(\half12)}=2^n,\quad \forall t\in\R{},
$$
where $n=\dim A$ and $\chi_c(t)$ is the characteristic polynomial of an idempotent $c$. In this paper, we only outline some general lines and discuss  basic properties of syzygies. The border-line case when an algebra contains 2-nilpotents and the exceptional case when there exists infinitely many idempotents will be considered elsewhere.

The paper is orginized as follows. We recall some well-known concepts  in \Cref{sec:prel}, then define the concept of a generic algebra and study its basic properties in \Cref{sec:generic}. We explain also here why the presence or absence of eigenvalue $\half12$ in the algebra spectrum plays an exceptional role. The principal syzygies are determined in \Cref{sec:syzbasic}  and some applications and  explicit examples are given in \Cref{sec:unital,sec:algdim2}. Furthermore, applying the  syzygy method we study in \Cref{sec:Peirce} some examples of algebras with a prescribed Peirce spectrum. A distinguished subclass  of nonassociative algebras is algebras admitting an associative bilinear form, the so-called metrised algebras. This class is in a natural correspondence  with the space of all cubic forms on the ground vector space. We discuss the spectral properties of metrised algebras in \Cref{sec:examples} and establish an extremal property of $\lambda=\half12$ in \Cref{sec:Except}. We also  show that the presence of $\half12$  in the algebra spectrum yields a fusion rule for the corresponding Peirce eigenspace.

\section{Preliminaries}\label{sec:prel}

In choosing what material to include here, we have tried to concentrate on the class of \textit{commutative} nonassociative algebras over a field of characteristic 0. In fact, many of our results, including the principal syzygies, still remain valid in the non-commutative case and for general finite fields but in that case some topics becomes lengthy and require more careful analysis, and will be treated elsewhere.

Therefore, in what follows by $A$ we mean an (always finite dimensional)  commutative nonassociative algebra over a filed $\field$. We point out that `algebra' always means a nonassociative algebra.

We need to make some additional assumptions on the ground field $\field$. If not explicitly stated otherwise, we shall assume that $\field$ is a subfield of $\mathbb{C}$, the field of complex numbers. By $A_\mathbb{C}$ we denote the complexification of $A$ obtained in an obvious way by extending of the ground field such that $\dim_\field A=\dim_\mathbb{C} A_\mathbb{C}$.

An element $c$ is called idempotent if $c^2=c$ and 2-nilpotent if $c^2=0$. By
$$
\Idm(A)=\{0\ne c\in A:c^2=c\}
$$
we denote the set of all \textit{nonzero}  idempotents of $A$ and the complete set of idempotents will be  denoted by
$$
\Idm_0(A)=\{0\}\cup \Idm(A).
$$
The  set of all idempotents and 2-nilpotents of $A$ will be denoted by
$$
\mathbf{P}(A)=\{x\in A: \text{either $x^2=x$ or $x^2=0$}\}
$$

If the algebra  $A$ is unital with  unit $e$ then given an idempotent $c\in \Idm_0(A)$, its conjugate $\bar c:=e-c$ is also idempotent:
$$
\bar c^2=(e-c)^2=e-2c+c=\bar c.
$$
It is also well known that $c$ and $\bar c $ are orthogonal in the sense that $c \bar c=0$.

We follow the standard notation and denote by $L_x$ the multiplication operator (sometimes also called adjoint of $x$):
$$
L_xy=xy=yx.
$$
Regarding $L_x$ as an endomorphism in the vector space $A$, we define the corresponding characteristic polynomial  by
$$
\chi_x(t)=\det (L_x-tI), \quad t\in \field.
$$
Let $\sigma(x)$ denote the set of (in general complex) roots of the characteristic equation $\chi_x(t)=0$ counting multiplicity. By the made assumption, $\sigma(x)$ is well defined and is said to be the \textit{Peirce spectrum} of $x$. It is easy to see that if a root $t\in \field$ then the corresponding \textit{Peirce subspace}
$$
A_c(t):=\ker (L_c-t I)
$$
is nontrivial. Thus, any $t\in \sigma(c)\cap \field$ is actually an eigenvalue of $L_x$.

Now suppose that $c\in \Idm_0(A)$ is a nonzero idempotent. Then  $t=1$ is an obvious eigenvalue of $L_c$ (corresponding to $c$), thus $1\in \sigma(c)$. Distinct elements of the Peirce spectrum $\sigma(c)$ are called \textit{Peirce numbers}.

An idempotent $c$ is called \textit{semisimple} if $A$ is decomposable as the sum of the corresponding  Peirce subspaces:
$$
A=\bigoplus_{i}A_c(\lambda_i),
$$
where $\lambda_i$ are the Peirce numbers of $c$. We define in this case the corresponding \textit{Peirce dimensions}
$$
n_c(\lambda)=\dim \ker (L_c-\lambda I).
$$

Note that the number of idempotents in a (finite-dimensional) algebra can not be very arbitrary. Namely, the set of idempotents can be studied by purely algebraic geometry methods, an idea coming back to the classical paper of Segre~\cite{Segre}. More precisely, Segre showed that the set $\mathbf{P}(A)$ can be described as the solution set of a system of quadratic equations over $\field$, actually as intersection of certain quadrics. This in particular implies that a real or complex algebra without nilpotent elements always admits idempotents.

For the following convenience we briefly recall Segre's argument. Let us consider an algebra over $\field$, not necessarily commutative. Let us associate to $A$ with the multiplication  map
$$
\psi_A(u,v)=uv:A\times A\to A
$$
which is naturally identified with a corresponding  element $\psi_A \in V^*\otimes V^*\otimes V$. If  $\mathbf{e}=\{e_1,\ldots,e_n\}$  is an arbitrary basis in $A$, where $n=\dim_\field  A$, then $\psi_A$ induces a $\field$-quadratic polynomial map $\Psi_A:\field^n\to \field^n$ defined by
\begin{equation}\label{psidef}
\psi_A\circ \epsilon=\epsilon \circ \Psi_A,
\end{equation}
where $\epsilon$ is the coordinatization map
$$
\epsilon(x):=\sum_{i=1}^nx_ie_i: \field^n\to A, \quad \quad x=(x_1,\ldots, x_n)\in \field ^n.
$$
In this setting, $\Psi_A$ is a bilinear map on $\field^n$. Then an element $c=\epsilon(x)\in A$ is idempotent if and only if the corresponding  $x\in \field^n$ is a fixed point of $\Psi_A(x,x)$, i.e.
\begin{equation}\label{FiX}
\Psi_A(x,x)-x=0.
\end{equation}
It is  convenient to consider the projectivization of the latter system. Namely, let
$$
\Psi^{\mathbf{P}}_A(X)=\Psi_A(x,x)-x_0x,
$$
where $X=(x_0,x_1,\ldots,x_n)\in \field^{n+1}$.
The modified equation
\begin{equation}\label{FX}
\Psi^{\mathbf{P}}_A(X)=0
\end{equation}
is homogeneous of degree 2. By the made assumption on $\field$, we can consider both \eqref{FiX} and \eqref{FX} as equations over the complex numbers. Furthermore, \eqref{FX} defines a variety in $\mathbb{C}\mathbb{P}^n$. Clearly, if $x$ solves \eqref{FiX} then $X=(1,x)$ is a solution of \eqref{FX}, and, conversely, $X=(x_0,x)$ solves \eqref{FX} with $x_0\ne 0$ then $\frac{1}{x_0}x$ is a solution of \eqref{FiX}. In the exceptional case $x_0=0$, one has $\Phi(x)=0$, i.e. $\epsilon(x)$ is a 2-nilpotent in $A$.

In summary, there exists a natural bijection (depending on a choice of a basis in $A$) between the set $\mathbf{P}(A_\mathbb{C})$ and all solutions of \eqref{FX} in $\mathbb{C}\mathbb{P}^n$. In this picture, 2-nilpotents correspond to the `infinite' part of solutions of \eqref{FiX} (i.e. solutions of \eqref{FX} with $x_0=0$).

Then the classical Bez\'out's theorem implies the following dichotomy: either there are infinitely many solutions of  \eqref{FX} or the number of distinct solutions is less or equal to $2^n$, where $n=\dim_\field A$. Therefore if the set $\mathbf{P}(A_\mathbb{C})$ is finite then necessarily
\begin{equation}\label{cardinality}
\mathrm{card} \,\mathbf{P}(A_\mathbb{C})\le 2^n
\end{equation}
We point out that one should interpret  a  solution to \eqref{FX} in the projective sense.

Some remarks are in order. First note that the above correspondence makes an explicit bijection between idempotents and 2-nilpotents only in the complexification $A_{\mathbb{C}}$. In general, if $X=(x_0,x)$ is a solution to \eqref{FX} then $x\in \mathbf{P}(A)$ only if $X\in \field^n$. This, of course, also yields the corresponding inequality over $\field$:
\begin{equation}\label{cardinality1}
\mathrm{card} \,\mathbf{P}(A_\field)\le \mathrm{card} \,\mathbf{P}(A_\mathbb{C})\le 2^n.
\end{equation}
Note, however, that a priori it is possible that there can exist only finitely many number solutions over $\field$ while there can be infinitely many solutions over $\mathbb{C}$.

\section{Generic nonassociative algebras and the exceptionality of $\frac{1}{2}$}\label{sec:generic}

It is well known that a generic (in the Zariski sense) polynomial system has B\'ezout's number of solutions. In our case, if $K=\mathbb{C}$ then an algebra having exactly $2^{\dim A}$ idempotents (B\'ezout's number for \eqref{FX}) is generic in the sense that the subset of nonassociative algebra structures on $V$  with exactly $2^{\dim A}$ idempotents is an open Zariski subset in $V^*\otimes V^*\otimes V$.
This motivates the following definition.

\begin{definition}\label{def1}
An algebra $A$ over $\field$ is called a \textit{generic nonassociative algebra}, or generic NA algebra, if its comlexification $A_{\mathbb{C}}$ contains exactly $2^n$ distinct idempotents, where $n=\dim A$.
\end{definition}

The definition given above should not be confused with  similar definitions of generic subsets for certain distinguished classes of algebras (like a generic division algebra). Namely, our definition distinguish generic algebras in the class of \textit{all} nonassociative algebras. We refer also to \cite{RohrlW}, p.~196, where the generic phenomenon is essentially interpreted as the absence of 2-nilpotents in the presence of idempotents which supports our definition.

Note that \Cref{def1} together with B\'ezout's theorem imply that if $A$ is generic then neither $A_{\mathbb{C}}$ nor $A$ have nonzero $2$-nilpotents.

The class of generic NA algebras  can be thought of as the most natural model for testing the above program. First note that the definition itself implies certain obstructions on the algebras spectrum. The following two criterium shows that the property being a generic for an algebra is essentially equivalent to the fact that the algebra spectrum does not contain $\half12$.

\begin{theorem}\label{the:12}
If $A$ is a commutative generic algebra  then $\half12\not\in\sigma(A)$. In the converse direction: if $\half12\not\in\sigma(A)$ and $A$ does not contain $2$-nilpotents then $A$ is generic.
\end{theorem}

\begin{proof}
First let us  define the associated  quadratic map
$$
\Psi_A(x):=\Psi_A(x,x):\field^n\to \field^n
$$
and consider the fixed point equation
\begin{equation}\label{FiX0}
f_A(x):=\Psi_A(x)-x=0.
\end{equation}
By the commutativity assumption, the multiplication map $$\Psi_A(x,y)=\Psi_A(y,x)$$ is symmetric, therefore it is recovered from $\Psi_A$ by polarization:
\begin{equation}\label{polar}
\Psi_A(x,y)=\half12(\Psi_A(x+y)-\Psi_A(x)-\Psi_A(y))
=\half12 D\Psi_A(x)\,y,
\end{equation}
in particular, this yields $\epsilon(L_xy)= D\Psi_A(x)\,y$ for all $x,y\in \field^n$, i.e.
\begin{equation}\label{polar1}
\epsilon\circ L_x=\half12 D\Psi_A(x).
\end{equation}
This yields that
\begin{equation}\label{detjac}
\begin{split}
\det (D\Psi_A(c)-I)&=\det (2\epsilon\circ L_x-I)\\
&=2^n\det (\epsilon\circ L_x-\half12I)\\
&=2^n\chi_c(\half12).
\end{split}
\end{equation}
All the corresponding relations above, of course, are valid as well for $D\Psi_{A_\mathbb{C}}$.

Now, suppose that  $A$ is generic. Since the number of idempotents in $A_{\mathbb{C}}$ is maximal (equal to B\'ezout's number $2^n$) and all idempotents are distinct, it follows that all solutions of \eqref{FiX0}  are regular points, see \cite[Sec.~8]{Fulton}, \cite[Sec.~4]{Shafarevich1}, therefore
\begin{equation*}\label{ne0}
\det (D\Psi_A(c)-I)\ne 0, \qquad \forall c\in \Idm(A),
\end{equation*}
therefore, it follows from \eqref{detjac} that for any idempotent $c\in \Idm(A)$: $\half12\not\in \sigma(c)$.

In the converse direction, let $\half12\not\in\sigma(A)$ and let $A_{\mathbb{C}}$ does not contain nonzero 2-nilpotents. Arguing by contradiction, let $A_{\mathbb{C}}$ have either (i) multiple idempotents or (ii) infinitely many idempotents. Then (i) and (ii)  are respectively equivalent to saying that equation \eqref{FiX0} has (i) multiple solutions and (ii) infinitely many solutions.

Now, if (i) holds then there is a multiple  solution $x$ of \eqref{FiX0} representing a multiple idempotent $c=\epsilon(x)\in \Idm(A_{\mathbb{C}})$. Then $\det f_A(x)=\det (D\Psi_A(c)-I)= 0$, thus \eqref{detjac} implies $\chi_c(\half12)=0$, a contradiction. Next, suppose that (ii) holds and let $E:=\{x_k\}_{1\le k\le \infty}\subset A_{\mathbb{C}}$ be a countable subset of distinct solutions of \eqref{FiX0}. Let us equip $A_\mathbb{C}$ with an Euclidean metric $\|x\|$. If the set $E$ is unbounded  then there exists a subsequence (we denote it by $x_k$ again) such that $x_{k}\to \infty$ as $k\to \infty$, therefore we have from \eqref{FiX0} that $\lim_{k\to\infty} \Psi_A(x_{k}/\|x_{k}\|)=0$. This proves by the standard compacteness argument that there exist a unit vector $y\in A_{\mathbb{C}}$, $\|y\|=1$ (an accumulation point of $x_{k}/\|x_{k}\|$) such that $\Psi_A(y)=0$, i.e. $y^2=0$ on the algebra level. The latter means that $y$ is a 2-nilpotent, a contradiction. Finally, if the sequence $x_k$ is bounded then one can find a finite accumulation point, say, $z\in A_{\mathbb{C}}$ which is the limit of a subsequence of $x_k$. Clearly, $z$ is a solution of \eqref{FiX0}, therefore $\epsilon(z)$ is an idempotent of $A_{\mathbb{C}}$. It also easily follows that $z$ is a \textit{non-isolated} solution of \eqref{FiX0}, hence
$$
0=\det (D\Psi_A(z)-I)=2^n\chi_z(\half12)
$$
a contradiction again. The theorem is proved.

\end{proof}

\begin{remark}
The proof of Proposition~\ref{the:12} is also valid for the noncommutative case. But in this case one should  require that the spectrum of the symmetrized multiplication $L_c+R_c$ does not contain $1$.
\end{remark}

It is interesting to point out here that the classical examples of nonassociative algebras like Jordan and power-associative algebras are non-generic: indeed they have $\half12$ in the Peirce spectrum. The same property share the Hsiang algebras mentioned in Introduction. Furthermore, it was recently remarked in \cite{Rehren15}, \cite{HRS15} that the classification of axial algebras depends very much on the inclusion $\half12\in\sigma(A)$.

It was already pointed out that the non-generic case is essentially equivalent to the inclusion $\half12\in \sigma(A_{\mathbb{C}})$ except for the case when $A_{\mathbb{C}}$ contains nonzero 2-nilpotents. The latter situation is still close to the generic case: indeed, one can prove that the syzygies in \Cref{th4} also valid with some mild restrictions. But the case $\half12\in \sigma(A_{\mathbb{C}})$ is really peculiar because in that case normally  $A_{\mathbb{C}}$ contains {\color{red}{multiple or}} infinite number of idempotents. In fact, it follows from  Bez\'out's theorem that $A_{\mathbb{C}}$ contains some varieties of idempotents.

\begin{remark}
In the case when an algebra $A$ over $\R{}$ or $\mathbb{C}$ admits a topological structure consistent with the multiplicative structure of $A$, it is also interesting to study the path connectivity between idempotents, see for example  \cite{Aupetit},\cite{Esterle}. It turns out that the inclusion $\half12\in\sigma(A)$ is also crucial here. In particularly, in \cite{Esterle}, J.~Esterle proves that two homotopic idempotents may always be connected by a polynomial idempotent-valued path.
\end{remark}

We illustrate the latter remark by the following simple observation. Let $A$ be a commutative algebra over $\R{}$ containing a smooth path of idempotents (homotopic idempotents), i.e. $c=c(t)\in \Idm(A)$, $t\in \Delta\subset \R{}$. Then differentiating $c^2(t)=c(t)$ with respect to $t$ yields $c(t)c'(t)=\half12 c'(t)$, thus $\half12\in \sigma(c(t))$ as long as  $c(t)$ is regular at $t$. In fact, a more strong property holds.

\begin{proposition}
\label{pro:conidm}
Let $A$ be a commutative  finite dimensional algebra over a field $\field$. If there are three idempotents $c_1,c_2\in \Idm(A)$ such that  $\alpha c_1+(1-\alpha) c_2\in \Idm(A)$ for some $\alpha\in K$ such that  $\alpha(1-\alpha) \neq 0$ then $\alpha c_1+(1-\alpha) c_2\in \Idm(A)$ for all $\alpha\in \field$ and $c_1-c_2\in \Nil(A)$. In particular, if $\Nil(A)=0$ then $\Idm(A)$ any three distinct nonzero idempotents spans a two-dimensional subspace.
\end{proposition}

\begin{proof}
We have
\begin{equation}\label{alpha}
(\alpha c_1+(1-\alpha) c_2)^2=\alpha c_1+(1-\alpha) c_2,
\end{equation}
therefore $2\alpha(1-\alpha)c_1c_2=\alpha(1-\alpha)(c_1+c_2)$, implying by the made assumption that $c_1+c_2=2c_1c_2$, or equivalently $(c_1-c_2)^2=0$, hence $c_1-c_2\in \Nil(A)$. It also follows that  \eqref{alpha} holds true for all $\alpha \in \field$, as desired.
\end{proof}

\section{Syzygies in generic NA algebras}\label{sec:syzbasic}
In this section we show that a commutative algebra cannot have an arbitrary spectrum.  More precisely, if $\half12\not\in \sigma(A)$ then there exists $n=\dim A$ nontrivial identities on $\sigma(A)$. This remarkable phenomenon sheds a new light on the spectral properties of many well-established examples. We  discuss these in more detail in  \Cref{sec:examples} below.

\subsection{The principal syzygies}

 We  need the following version of the celebrated Euler-Jacobi formula which gives an algebraic relation between the critical points of a polynomial map and their indices, see \cite[p.~106]{ArnVarG} (see also Theorem~4.3 in \cite{BKK12}).

\begin{theorem*}[Euler-Jacobi Formula]
Let $F(x) = (F_1(x),\ldots, F_n(x))$ be a polynomial map and let $\widetilde{F}$ be the polynomial map, whose components are the highest homogeneous terms of the components of $F$. Denote by $S_{\mathbb{C}}(F)$ the set of all complex roots of $F_1(x)=F_2(x)=\ldots=F_n(x)=0$ and suppose that any root $a\in S_{\mathbb{C}}(F)$ is simple and, furthermore,  that $S_{\mathbb{C}}(\widetilde{F})=\{0\}$. Then, for any polynomial $h$ of degree less than the degree of the Jacobian: $\deg h<N=-n+\sum_{i=1}^n\deg F_i$, one has
\begin{equation}\label{EuJa}
\sum_{a\in S(F)}\frac{h(a)}{\det[DF(a)]}=0
\end{equation}
where $D(\cdot)$ denotes the Jacobi matrix.
\end{theorem*}

Now, let $A$ be a commutative nonassociative algebra over $K$. Using the notation of \Cref{sec:prel}, associate to the multiplicative structure on $A$ the bilinear map $\Psi_A$ by \eqref{psidef}
such that the multiplication in the algebra $\epsilon$-conjugates with the Jacobi map: $\epsilon\circ L_x=\half12 D\Psi_A.$ In this setting, the coordinatization $x=\epsilon(c)$ of an arbitrary idempotent $c\in \Idm(A)$ is a fixed point of the quadratic map $\Psi_A(x)$ and vice versa, any fixed point of $\Psi_A(x)$ gives rise to an idempotent of $A$.
Then in the notation of the Euler-Jacobi Formula and \eqref{FiX0} we have
\begin{equation}\label{idempS}
\epsilon(\Idm(A))=S_K(f_A).
\end{equation}
Similarly, the set 2-nilpotents of $A$ coincides with the set of solutions of the reduced system $\widetilde{f}_A\equiv \Psi_A$:
$$
\epsilon(\Nil(A))=S_K(\widetilde{f}_A)=S_K(\Psi_A).
$$
Furthermore, we have from  \Cref{the:12} that an idempotent $c\in \Idm(A)$ is a regular point of the map $f_A$ if and only if
\begin{align}\label{nondeg}
\det Df_A(c)=2^n\chi_c(\half12)\ne0.
\end{align}

 Now we are ready to prove the main result of this section.

\begin{theorem}\label{th4}
Let $A$ be a generic commutative nonassociative algebra over $K$, $\dim A=n$.  Then
\begin{equation}\label{polynom}
\sum_{c\in \Idm_0 (A)}\frac{\chi_c(t)}{\chi_c(\half12)}=2^n,\quad \forall t\in\R{}.
\end{equation}
In particular,
\begin{align}
\sum_{c\in \Idm_0 (A)}\frac{\chi^{(k)}_c(\half12)}{\chi_c(\half12)}&=0, \quad k=1,2,\ldots, n\label{EuJa1}
\end{align}
where $\chi^{(k)}$ denotes the $k$th derivative of $\chi$.
\end{theorem}

\begin{proof}
In  notation of the Euler-Jacobi Formula, we have $F(x)=f_A(x)$, $\widetilde{F}(x)=\Psi_A(x)$. Since $A$ is generic, it has exactly $2^n$ distinct idempotents, thus they are all regular points of $f_A(x)$, in particular, \eqref{nondeg} holds for any $c\in \Idm(A)$. Since $A$ is generic, we also have from \eqref{idempS}
$$
S_K(f_A)=S_{\mathbb{C}}(f_A)=\epsilon(\Idm(A)) \quad \text{and}\quad
S_K(\widetilde{f}_A)=\epsilon(\Nil(A))=\{0\}.
$$
Furthermore, the condition on $h$ reads in the present notation as
$$
\deg h<N=-n+\sum_{i=1}^n\deg F_i=-n+2n=n.
$$
Therefore, combining the Euler-Jacobi Formula with \eqref{nondeg}, we obtain for any polynomial $h$ of degree $\le n-1$ in the variables $x_1,\ldots,x_n$ that
\begin{equation}\label{det1}
0=\sum_{c\in \Idm(A)}\frac{h(x_c)}{\det[Df_A(x_c]}=
\frac{1}{2^n}\sum_{c\in \Idm(A)}\frac{h(x_c)}{\chi_c(\half12)}
\end{equation}
where $x_c\in \field^n$ is defined by $\epsilon(x_c)=c$, and $c$ runs over all idempotents in $\Idm(A)$.

Let us rewrite the shifted characteristic polynomial as follows:
$$
(-1)^n\chi_c(t-\half12)=t^n-a_1t^{n-1}+\ldots+(-1)^na_n, \qquad a_k=a_k(c).
$$
Then each $a_k$ is an elementary  symmetric function of the roots $t_1,\ldots,t_n$ of $P(t)$. By the Newton's identities, the coefficient $a_k$ is also expressible as a linear combination of power sums
$$
p_i=p_i(c)=t_1^i+\ldots+t_n^i.
$$For example,
\begin{align*}
a_1&=T_1(p_1):=p_1\\
a_2&=T_2(p_1,p_2):=\half12(p_1^2-p_2)\\
a_3&=T_3(p_1,p_2,p_3):=\half16(p_1^3-3p_1p_2+2p_3),\ldots
\end{align*}
Each polynomial $T_k(p_1,\ldots, p_s)$ has homogeneous degree $s$ in the sense that all monomials $p_1^{m_1}\cdots p_k^{m_k}$ in $T_k$ has the total degree $k=m_1+2m_2+\ldots+km_k$.

Next,  the power sums can be evaluated as the successive traces of $L_c-\half12 I$:
$$
p_k(c)=t_1^k+\ldots+t_n^k=\trace (L_c-\half12I)^k=:\tau_k.
$$
Therefore, $a_k=T_k(\tau_1,\ldots, \tau_k)$.
Now, let us define $h(x)$ in the Euler-Jacobi Formula above by
$$
h_k(x)=T_k(\trace Df_A(x),\ldots, \trace (Df_A(x))^k)
, \quad 0\le k\le n-1.
$$
Note that the entries of the Jacobi matrix $Df_A(x)$ are linear functions in the variables $x_i$, thus
$\deg h_k=k,$
which  is consistent with the  degree condition in  the Euler-Jacobi Formula for all $0\le k\le n-1$. By \eqref{polar1} and the homogeneity we have
$$
h_k(x_c)=T_k(\trace Df_A(x_c),\ldots, \trace (Df_A(x_c))^k)=2^k a_k(c)
$$
therefore applying \eqref{det1} we obtain
\begin{equation}\label{det1'}
\sum_{c\in \Idm(A)}\frac{h_k(x_c)}{\chi_c(\half12)}=\sum_{c\in \Idm(A)}\frac{a_k(c)}{\chi_c(\half12)}=0, \qquad 0\le k\le n-1.
\end{equation}
Since $a_k(c)=b_k\chi_c^{(k)}(\half12)$, where $b_k=(-1)^{k}/(n-k)!$ does not depend on $c$, we derive the identities for the derivatives \eqref{EuJa1}. Also, using Taylor's expansion $\chi_c(t)=\sum_{k=0}^n\frac{1}{k!}\chi_c^{(k)}(\half12)(t-\half12)^k$ yields \eqref{polynom}.
\end{proof}

\Cref{th4} describes the so-called symmetric  syzygies, i.e. when the numerator in \eqref{polynom} is a symmetric function of eigenvalues of each $c$. It is also convenient to have general scalar and vector syzygies. These given in the proposition below.

\begin{proposition}
Under conditions of \Cref{th4}, let $H(x):\field^n\to \field^s$ be a vector-valued polynomial map ($s\ge 1$) such that for each coordinate $\deg H_i\le n-1$, $1\le i\le n$. Then
\begin{align}
\sum_{c\in \Idm_0 (A)}\frac{H(x_c)}{\chi_c(\frac12)}&=0,\label{EuJa4}
\end{align}
where  $x_c\in \field^n$ is defined by $\epsilon(x_c)=c$.
In particular,
\begin{align}
\sum_{c\in \Idm(A)}\frac{c}{\chi_c(\frac12)}&=0,\label{EuJa2}
\end{align}

\begin{proof}
The first identity is just a corollary of \eqref{det1}. To prove \eqref{EuJa2}, we apply \eqref{EuJa4} for $H(x)=x$ followed by homomorphism $\epsilon$.
\end{proof}

\end{proposition}

\begin{corollary}
\label{cor:n-1}
Under conditions of \Cref{th4}
\begin{equation}\label{idemm}
\sum_{c\in \Idm_0 (A)}\frac{\chi_c(t)}{\chi_c(\half12)}=2^n(1-t^n)
\end{equation}
and
\begin{equation}\label{idemm1}
\sum_{c\in \Idm (A)}\frac{\widetilde{\chi}_c(t)}
{\widetilde{\chi}_c(\half12)}=2^{n-1}(1+t+ \ldots+t^{n-1}), \qquad \text{where}\quad
\widetilde{\chi}_c(t)=\frac{\chi_c(t)}{t-1}.
\end{equation}
\end{corollary}

\begin{proof}
Since $\chi_0(t)=t^n$, \eqref{polynom} follows from \eqref{idemm}. Next, since $1\in \sigma(c)$ for all idempotents $c$, one can factorize
$
\chi_c(t)=(t-1)\widetilde{\chi}_c(t)
$
so that \eqref{idemm} yields \eqref{idemm1}.
\end{proof}

\begin{remark}\label{rem:num}
Some remarks concerning the number of independent syzygies is in order. Note first that we do not study this question in details because it requires a more careful analysis even in the generic case. Formally, it may be thought that the number of syzygies $S$ is the degree of the polynomial identity in \eqref{polynom} minus the tautological identity obtained when $t=\half12$, i.e. $S=n-1$. This is true, for example, for generic two-dimensional algebras, see \Cref{sec:algdim2} below. In fact, the number of nontrivial syzygies is sometimes less than $n-1$, see the discussion of unital algebras in the next section. In fact, any a priori assumption on the algebra structure such as the existence of unity, an algebra identity etc, of course, decreases the number of possible (`extra') syzygies defined by \eqref{polynom} or \eqref{idemm1}. This question deserves a separate study.
\end{remark}

\subsection{Syzygies in unital generic algebras}\label{sec:unital}
Let us consider a unital commutative algebra $A$. Then the unit $e$ is also a (nonzero) idempotent. In fact, as we shall see below, the existence of a unit decrease the number of nontrivial syzygies. This follows from the fact that the spectrum of each idempotent in a unital algebra is partially prescribed. Indeed, first note that there is a natural involution map on the set of idempotents in the algebra $A$: $
\bar c:=e-c$ is an idempotent if and only $c$ is (the idempotent $\bar c$ is called the conjugate to $c$). Then
$$
c\bar c=c(e-c)=c-c=0,
$$
i.e. each  nontrivial (i.e. distinct from the unit and the zero elements) idempotent has at least the eigenvalues $1$ and $0$ in its spectrum:
\begin{equation}\label{01}
\{0,1\}\subset \sigma(c).
\end{equation}
Furthermore, if $\dim A=n$ then the corresponding  characteristic polynomials obviously related as follows:
\begin{equation}\label{conjue}
\chi_{\bar c}(t)=(-1)^n\chi_c(1-t).
\end{equation}
For example, $\chi_0(t)=t^n$ and $\chi_{e}(t)=(t-1)^n$.

Suppose now that $A$ is generic. Then it has exactly $2^n$ distinct idempotents (including the zero and the unit elements). Observe that $\bar c\ne c$ because otherwise $c=c^2=c\bar c=0$ implying $c=0$, and on the other hand, $c=\bar c=e-c=e$, a contradiction. Thus, the conjugation $c\to \bar c$ splits up the set of all idempotents $\Idm_0(A)$ into $2^{n-1}$ \textit{distinct} pairs of idempotents. Let $$
\Idm^+(A):=\{c_0=0,c_1,\ldots,c_{2^{n-1}-1}\}
$$
be set of some representees of the pairs (of course, this choice is not unique).

\begin{proposition}
\label{pro:unital}
Let $A$ be a unital generic algebra of dimension $n\ge 2$ and let $\Idm^+(A)$ be set of some representees of the pairs. Then
\begin{equation}\label{P1}
\sum_{c\in \Idm^+ (A)}\frac{\chi_c(\half12+s) +\chi_c(\half12-s)}{\chi_c(\half12)}=2^n.
\end{equation}
\end{proposition}

\begin{proof} We have from \eqref{conjue} for any $c\in \Idm^+(A)$ that
\begin{equation}\label{conjueS}
\chi_{\bar c}(\half12+s)=(-1)^n\chi_c(\half12-s),
\end{equation}
and therefore
$$
\chi_{\bar c}(\half12)=(-1)^n\chi_{c}(\half12).
$$
Therefore
\begin{equation*}
\begin{split}
P_A(s)&:=-2^n+\sum_{c\in \Idm_0 (A)}\frac{\chi_c(t)}{\chi_c(\half12)}\\
&\,\,=
-2^n+\sum_{c\in \Idm^+ (A)}\frac{1}{\chi_c(\half12)}(\chi_c(\half12+s) +\chi_c(\half12-s))
\end{split}
\end{equation*}
is an \textit{even} polynomial. Also, according to \eqref{polynom}  $P(s)\equiv 0$, as desired.
\end{proof}

Below we consider some applications for small dimensions.
First suppose that $n=2$. Then $\Idm^+ (A)=\{c\}$, where $\Idm_0(A)=\{0,e,c,\bar c\}$, and $\chi_c(t)=(t-1)(t-\lambda)$, i.e. $\chi_c(\half12+s)=(s-\half12)(s+\half12-\lambda)$. Applying \eqref{P1} yields
$$
\frac{\chi_c(\half12-s)+ \chi_c(\half12+s)}{\chi_{c}(\half12)}=4- \frac{\chi_0(\half12-s)+ \chi_0(\half12+s)}{\chi_{0}(\half12)}\equiv 2(1-4s^2),
$$
therefore $8s^2\lambda=0$, hence $\lambda=0$ and $\chi_c(t)=(t-1)t$. The latter conclusion can also easily be derived directly from \eqref{01}.

Next consider the case $n=3$. Then $\Idm^+ (A)=\{0,c_1,c_2,c_3\}$ and  \eqref{P1} yields
\begin{equation}\label{halfa3}
\sum_{i=1}^3\frac{\chi_{c_i}(\half12-s)+ \chi_{c_i}(\half12+s)}{\chi_{c_i}(\half12)}=8- \frac{\chi_0(\half12-s)+ \chi_0(\half12+s)}{\chi_{0}(\half12)}\equiv 6(1-4s^2).
\end{equation}
Since $n=3$ and \eqref{01}, the characteristic polynomials of $c_i$ is $\chi_{c_i}(t)=t(t-1)(t-\alpha_i)$, $i=1,2,3$, $\alpha_i\in \mathbb{C}$. An easy analysis shows that \eqref{halfa3} holds identically.

This shows that  \textit{there are no nontrivial  syzygies (on the eigenvalues) in a 3-dimensional unital algebra}.
See also \eqref{Matsuo1} below an example of a 3-dimensional Matsuo algebra whose algebra spectrum is $\sigma(A)=\{0,1,\alpha,1-\alpha\}$,  $\alpha\in \mathbb{C}$.

Finally, consider $n=4$. Then $\Idm^+ (A)=\{0,c_1,c_2,c_3,c_4,c_5,c_6,c_7\}$ and $\chi_{c_i}(t)=t(t-1)(t-\alpha_i)(t-\beta_i)$, $1\le i\le 7$, $\alpha_i,\beta_i\in \mathbb{C}$, hence \eqref{P1} yields
\begin{equation}\label{halfa4}
\sum_{i=1}^7\frac{\chi_{c_i}(\half12-s)+ \chi_{c_i}(\half12+s)}{\chi_{c_i}(\half12)}= 2(1-4s^2)(4s^2+7).
\end{equation}
Note that the left hand side \eqref{halfa4} is an even degree polynomial, therefore \eqref{halfa4} implies totally three identities. But \eqref{halfa4} is satisfied identically for $s=\half12$ and $s=0$. Therefore there exists  only one nontrivial syzygy. One can, for example, equate the coefficients of $s^4$ in \eqref{halfa4}, which yields
\begin{equation}\label{half41}
\sum_{i=1}^7\frac{1}{(\half12-\alpha_i)(\half12-\beta_i)}=4.
\end{equation}

\subsection{Spectral theory of two-dimensional algebras}\label{sec:algdim2}
The algebras in dimension two are well-understood and classified, see for example \cite{Walcher1}, \cite{Peter2000}, \cite{PeterSch}, \cite{Diet05}, \cite{Darpo06}. Below we revisit two-dimensional commutative algebras with  emphasis on the idempotent and syzygies aspects. Our main goal is to show that in the two-dimensional case there exists essentially one nontrivial syzygy which agrees perfectly with \Cref{rem:num} above.

We want to point out that in that case it is possible to derive  the principal syzygy \eqref{polynom} by pure algebraic argument, without resorting to Bez\'out's theorem, and a filed $\field$ of any characteristic (non necessarily algebraically closed).

\begin{proposition}
\label{pro:nil}
If $A$ is a nonassociative commutative algebra over a field $\field $,  $\dim_\field  A=2$. Let $c_1\ne c_2$ and $c_i\in \Idm(A)$. Then either of the following holds:
\begin{itemize}
\item[(i)]
$\half12\in\sigma(c_i)$ for some $i=1,2$;
\item[(ii)]
there exists exactly three distinct nonzero idempotents in $A$;
\item[(iii)] there exists exactly two distinct nonzero idempotents $c_1$ and $c_2$ and a nonzero $2$-nilpotent;
\end{itemize}
\end{proposition}

\begin{proof}
Since $c_1\ne c_2$, they are linearly independent, thus form a basis in $A$. Write $c_1c_2=\alpha c_1+\beta c_2$, $\alpha, \beta\in \field$. Then $\sigma(c_1)=\{1,\beta\}$ and $\sigma(c_2)=\{1,\alpha\}$.

Let us assume that (i) does not hold, i.e. $\alpha,\beta\ne \half12$.
Consider $u=xc_1+y c_2$, $x,y\in \field$. Then
\begin{equation}\label{ue}
u^2=x(x+2\alpha y) c_1+y(2\beta x+y)c_2.
\end{equation}
If the determinant
$
\left|
  \begin{array}{cc}
    1 & 2\alpha \\
    2\beta & 1 \\
  \end{array}
\right|=1-4\alpha\beta= 0
$
then $\beta=\frac{1}{4\alpha}$, hence $c_1c_2=\alpha c_1+\frac{1}{4\alpha} c_2$, therefore
$$
(c_1-\half{1}{2\alpha}c_2)^2=0.
$$
Therefore $0\ne c_1-\half{1}{2\alpha}c_2\in \Nil(A)$. Let us show that there are not other idempotents in $A$ except for  $c_1$ and $c_2$. Indeed, if $u$ is such an idempotent then $xy\ne 0$ (otherwise  $u=c_1$ or $u=c_2$), hence \eqref{ue} yields
\begin{equation}\label{syss}
\left\{
\begin{array}{rl}
x+2\alpha y&=1\\
2\beta x+y&=1
\end{array}
\right.
\end{equation}
thus, using $\beta=\frac{1}{4\alpha}$ implies $2\alpha=1$, a contradiction. Therefore one comes to (iii).

If the determinant is nonzero: $\Delta:=1-4\alpha\beta\ne0$ then there exist a solution
$$
(x,y)=(\frac{1-2\alpha}{\Delta}, \frac{1-2\beta}{\Delta})
$$ of \eqref{syss}
which implies $u\in A$ such that $u^2=u$. Note that by the assumption$(1-2\alpha)(1-2\beta)\ne0$, hence $xy\ne 0$, i.e.   $u$ is  distinct from $c_1,c_2$. Therefore, we have three distinct idempotents, i.e. (ii). It also follows that in that case there exists exactly three idempotents.
\end{proof}

Now, let us consider the case of a \textit{generic} algebra of dimension 2.

\begin{theorem}\label{th:3lam}
Let $A$ be a nonassociative commutative algebra over a field $\field $ with $\dim_\field  A=2$. Suppose that there exists three distinct nonzero idempotents $c_i$, $i=1,2,3$. Then
\begin{equation}\label{syzygies}
4\lambda_1\lambda_2\lambda_3- \lambda_1-\lambda_2-\lambda_3+1=0,
\end{equation}
where $\sigma(c_i)=\{\lambda_i\}$.
\end{theorem}

\begin{proof}
Any pair of distinct idempotents is a basis of $A$. Decompose $c_ic_j=xc_i+yc_j$ for $i\ne j$. Then $\sigma(c_i)=\{1,y\}$ and $\sigma(c_j)=\{1,x\}$. This yields
\begin{align}
c_1c_2=\lambda_{2}c_1+\lambda_{1}c_2,\label{c12}\\
c_2c_3=\lambda_{3}c_2+\lambda_{2}c_3,\label{c23}\\
c_3c_1=\lambda_{1}c_3+\lambda_{3}c_1,\label{c31}
\end{align}
for some $\lambda_i\in \field$. In particular, this implies that  each idempotent is semi-simple and $\sigma(c_i)=\{1,\lambda_i\}$, $i=1,2,3$.

Next, by our assumption $c_1,c_2,c_3$ are distinct, hence
\begin{equation}\label{c123}
c_3=a_1c_1+a_2c_2,\qquad a_1,a_2\in \field,\,\, a_1a_2\ne 0,
\end{equation}
hence substituting the latter identity in \eqref{c31}
$$
(a_1c_1+a_2c_2)c_1=\lambda_{1}(a_1c_1+a_2c_2)+\lambda_{3}c_1,
$$
 one finds by virtue of \eqref{c12} that
$
(a_1(1-\lambda_1)+a_2\lambda_2-\lambda_{3})c_1=0
$
implying
\begin{equation}\label{lam1}
\lambda_{3} =(1-\lambda_1)a_1+\lambda_2 a_2.
\end{equation}
Arguing similarly with \eqref{c23} one arrives at
\begin{equation}\label{lam2}
\lambda_{3} =\lambda_1 a_1+(1-\lambda_2)a_2.
\end{equation}
This yields
$$
a_1(1-2\lambda_1)=a_2(1-2\lambda_2).
$$
Let first $\lambda_1=\frac12$. Then by the assumption $a_2\ne0$, hence $\lambda_2=\frac12$, which obviously satisfies \eqref{syzygies} for any $\lambda_3$. Next, if $\lambda_1\ne\frac12$ then $\lambda_2\ne\frac12$ and we have
\begin{equation}\label{ratio}
\frac{a_2}{a_1}=\frac{1-2\lambda_1}{1-2\lambda_2}.
\end{equation}
On the other hand, rewriting \eqref{c123} as
$$
c_1=-\frac{1}{a_1}c_3+\frac{a_2}{a_1}c_2= b_3c_3+b_2c_2,
$$
we obtain for symmetry reasons that
$
-a_2=\frac{b_2}{b_3}=\frac{1-2\lambda_3}{1-2\lambda_2},
$
hence from \eqref{ratio} we also have
$
-a_1=\frac{1-2\lambda_3}{1-2\lambda_1}.
$
Summing up \eqref{lam1} and \eqref{lam2} we obtain
\begin{equation}\label{sysygies2}
2\lambda_3=a_1+a_2=-(1-2\lambda_3)\left( \frac{1}{1-2\lambda_1}+\frac{1}{1-2\lambda_2}\right)
\end{equation}
which readily yields \eqref{syzygies}. The proposition follows.
\end{proof}

\begin{remark}
For any triple $\lambda_1,\lambda_2,\lambda_3$ satisfying \eqref{syzygies}, it is easy to construct an algebra with three idempotents having the spectrum $\sigma(c_i)=\{1,\lambda_i\}$. The relation \eqref{syzygies} (as well as \eqref{lambda0} below) appears in the classification of rank three algebras by S.~Walcher, see \cite[p.~3407]{Walcher1}.
\end{remark}

Now we show that  \eqref{syzygies} is actually equivalent to the principal syzygies \eqref{EuJa1}. To this end, note that if $\lambda_i\ne\half12$ then \eqref{sysygies2} yields also another form of \eqref{syzygies}, namely
\begin{equation}\label{syzygies0}
\frac{1}{1-2\lambda_1}+\frac{1}{1-2\lambda_2}+\frac{1}{1-2\lambda_3}=1.
\end{equation}

Taking into account that $\chi_{c_0}(t)=t^2$, where $c_0=0$ is the trivial idempotent, the latter equation can be written as
\begin{equation}\label{lambda0}
\frac{1}{{\chi}_{c_0}(\half12)}+
\frac{1}{{\chi}_{c_1}(\half12)}+
\frac{1}{{\chi}_{c_2}(\half12)}+
\frac{1}{{\chi}_{c_3}(\half12)}=0.
\end{equation}
This yields the syzygy in \eqref{EuJa1} for $k=n=2$. Another (the last for $n=2$) possibility is $k=1$ when  \eqref{EuJa1} becomes
\begin{equation}\label{lambda1}
\sum_{i=0}^3\frac{{\chi'}_{c_i}(\half12)}{{\chi}_{c_i}(\half12)}=0.
\end{equation}
Since ${\chi'}_{c_i}(\half12)=-\lambda_i$ for $i\ne 0$ and ${\chi'}_{c_0}(\half12)=1$, one readily verifies that \eqref{lambda1} is in fact equivalent to \eqref{lambda0}.

We have two further corollaries of \eqref{lambda0}.

\begin{corollary}
\label{cor:1}
Let $A$ be a nonassociative commutative algebra over a field $\field $, $\dim_\field  A=2$, and let $c_i$, $i=1,2,3$ be three nonzero idempotents. If $\frac12\in\sigma(c_i)$ for some $i$ then at least one of the remained idempotents has the same property.
\end{corollary}

\begin{proof}
Indeed, let $S(\lambda_1,\lambda_2,\lambda_3)$ denote the left hand side of \eqref{syzygies}. Then $S$ is irreducible in $\mathbb{C}[\lambda_1,\lambda_2,\lambda_3]$. But  if $\lambda_i=\frac12$ for some $i$, say $\lambda_3=\frac12$ then $S$ factorizes as follows:
$$
S(\lambda_1,\lambda_2,\frac12)=
\frac12(2\lambda_1-1)(2\lambda_2-1)
$$
This yields the desired conclusion.
\end{proof}

\begin{corollary}
 Given two linearly independent idempotents $e_1,e_2$ in a two dimensional non associative algebra $A$ over a field $K$, its
 multiplication table contains explicitly the Peirce numbers of $e_1,e_2$.
\end{corollary}

   \begin{proof} It follows immediately from \eqref{c12}--\eqref{c31}.
   \end{proof}

\begin{remark}
Such a multiplication table structure may be recognized as a {\it diagonal}, where the spectral parameters are presented explicitly.
\end{remark}

\section{Algebras with a prescribed Peirce spectrum}\label{sec:Peirce}
The spectrum of any nonzero idempotent contains $1$. If algebra is unital then the spectrum of any nontrivial idempotent contains $0$ and $1$ (see section~\ref{sec:unital}). There are many algebras sharing a remarkable property: the spectrum of all or a `large' subset of idempotents is constant or contains some prescribed values. It is interesting to know which common properties such algebras have.

\subsection{Algebras with constant or nearly constant spectrum}
The first natural and nontrivial example which can be treated by virtue of the constructed syzygies are generic algebras with constant spectrum.

\begin{definition}
An algebra is said to have a constant spectrum if all nonzero idempotents have the same spectrum (counting multiplicities).
\end{definition}

One such family is the so-called Hsiang (or REC algebras, see \cite[sec.~6]{NTVbook}). More precisely, a Hsiang algebra $A$ is a commutative algebra over $\R{}$ with a symmetric bilinear form $\scal{}{}$ (see the definition \eqref{Qass} below) such that for any element $x\in A$  two following identities hold:
$$
\scal{x^2}{x^3}=\scal{x}{x^2}\scal{x}{x}, \quad \trace L_x=0.
$$
It can be proved that the spectrum of \textit{any} idempotent in $A$ is constant and contains only the eigenvalues $\pm1$ and $\pm\half12$ (with certain multiplicities independent on a choice of an idempotent). However, since $\half12\in \sigma(c)$, Hsiang algebras are \textit{not} generic (any Hsiang algebra contains infinitely many idempotents). Therefore, the above syzygies are not applicable  here.

Nevertheless, in two dimensions an algebra with constant spectrum can be easily constructed. Let us consider a commutative two-dimensional algebra generated by $e_1$ and $e_2$ with identities
$$
e_1^2=-e_2^2=e_1, \quad e_1e_2=e_2e_1=-e_2.
$$
A simple analysis reveals that there is exactly $4=2^2$ idempotents:
$$
c_0=0, \quad c_1=e_1, \quad \text{and}\,\,\, c_{2,3}=-\half12 e_1\pm \half{\sqrt{3}}{2} e_2.
$$
In particular, $A$ is a generic algebra over $\R{}$. It is easily verified that the algebra $A$ possesses the constant spectrum property: the spectrum of any \textit{nonzero} idempotent is the same:
$$
\sigma(c_i)=\{1,-1\}, \quad i=1,2,3.
$$
One can readily prove that in dimension 2 any commutative algebra with the above property is necessarily isomorphic to $A$.

In the general case, one has the following observation.

\begin{corollary}\label{cor:same}
If  $A$ is a generic algebra, $n=\dim A$, such that all nonzero idempotents has the same spectrum then for any $c\in \Idm_0(A)$, $\chi_c(t)=(-1)^n(t^n-1)$. In other words, if $A$ is such an algebra then
\begin{equation}\label{prim}
\sigma(c)=\{e^{\frac{2\pi k\sqrt{-1}}{n}}, \quad k=0,1,2,\ldots,n-1\}
\end{equation}
for any idempotent $c\ne 0$.
\end{corollary}

\begin{proof}
Since the characteristic polynomial $\chi_c(t)$ is the same for all nonzero idempotents $c$ then using \eqref{idemm} we obtain that
$$
\frac{\chi_c(t)}{\chi_c(\half12)}=\frac{2^n(1-t^n)}{2^n-1}
$$
and the desired conclusion follows immediately from the last identity.
\end{proof}

It is interesting to know whether an algebra satisfying the conditions of \Cref{cor:same} realizable for any $n\ge 3$. The following example shows that this holds at least for $n=3$.

\begin{example}
 Let us consider a three dimensional commutative algebra over $\mathbb{C}$ generated by \textit{idempotents} $c_1, c_2,c_3$ with the multiplication table
\begin{equation*}
  c_ic_j=\alpha c_i+\beta c_j+\gamma c_k, \qquad \text {where $(i,j,k)$  is a \textit{cyclic} permutation of  $(1,2,3)$},
\end{equation*}
with the structure constants $$\alpha,\beta=-\half12\pm\half14\sqrt{-6+2\sqrt{-7}},\quad \gamma=\half14-\half14 \sqrt{-7}$$
The algebra $A$ is generic because, except for the basis idempotents $c_i$, there exists exactly $4=7-3$ nonzero idempotents, namely
\begin{align*}
  c_4&=-\gamma(c_1+c_2+c_3)\\
  c_5&=(\gamma-1)c_1-\gamma c_2+\gamma c_3,\\
  c_6&=\gamma c_1+(\gamma-1)c_2-\gamma c_3,\\
  c_7&=-\gamma c_1+\gamma c_2+(\gamma-1)c_3.
\end{align*}
A straightforward verification shows that all idempotents have the same spectrum
$$
\sigma(c_i)=\{1,\half12(-1\pm\sqrt{-3}\}, \quad \forall c_i\in \Idm_0(A).
$$
Note that the spectrum is constant and its elements are exactly the three roots of $z^3-1=0$ given in \eqref{prim}. Furthermore, it can be seen the validity of syzygies \eqref{EuJa2}:
\begin{equation*}
  \sum_{i=1}^{7} c_i=0.
\end{equation*}
One can prove that any three-dimensional generic algebra satisfying conditions of \Cref{cor:same} is necessarily isomorphic to the above algebra.
\end{example}

\begin{remark}
Dropping the requirement that the algebra $A$ is generic, yields many other families of nonassociative  algebras with constant spectrum. See for example an algebra with finitely many idempotents given in \Cref{sec:Hscub} below.
\end{remark}

Let us relax the constant spectrum property and consider a generic  algebra such that all idempotents has a common value in the spectrum.

\begin{corollary}\label{cor:common}
Let $A$ be a generic algebra. If $\alpha\in \sigma(c)$ for all nonzero $c\in \Idm_0(A)$ then $\alpha^n=1$ and $\alpha\ne 1$.
\end{corollary}

\begin{proof}
If $\alpha\ne1$ is a common value of the spectrum $\sigma(c)$ for all $c\ne 0$ then it is a common root of all $\widetilde{\chi}_c(t)$ in \eqref{idemm1}, hence $t-\alpha$ divides the right hand side of \eqref{idemm1}, implying the desired conclusion.
\end{proof}

\subsection{Algebras with a thin spectrum} \label{sec:3Ca}
It was already mentioned in Introduction, that many interesting examples of nonassociative algebras share another characteristic property: the spectrum of each idempotent consists  only of few distinct prescribed eigenvalues. Loosely speaking, such  an algebra has a `thin' spectrum.

As an example, let us consider the Matsuo algebra $3C_\alpha$. This is an particular example of the so-called Matsuo algebras family \cite{Matsuo05} appearing in the context of the Griess algebras, see \cite{Matsuo05}, \cite{Rehren17}. More precisely,  $A=3C_\alpha$ is the three-dimensional algebra over a field $\field$ containing $\alpha$ and $\mathrm{char}(\field)\ne 2$ spanned by three idempotents $e_1,e_2,e_3$ subject to the algebra identities
\begin{equation}\label{Matsuo1}
e_ie_j=\half{\alpha}{2}(e_i+e_j-e_k), \quad \{i,j,k\}=\{1,2,3\}.
\end{equation}
Then a simple analysis reveals that if $\alpha\ne -1$  and $\alpha\ne \half12$ then there exists exactly $7=2^3-1$ distinct nonzero idempotents, namely
\begin{align*}
e_7\,\,\,\,&=\half{1}{\alpha+1}(e_1+e_2+e_3),\\
e_{3+i}&=\bar e_i, \quad i=1,2,3,
\end{align*}
where $e_7$ is the algebra unit and $\bar c=e_7-c$ is the conjugate idempotent. In summary we have (see also \cite{Rehren15})

\begin{proposition}
If $\alpha\ne -1, \half12$ then the Matsuo algebra $3C_\alpha$ is a $3$-dimensional generic unital algebra. Its spectrum is given as follows:
\begin{align*}
\sigma(e_0)&=\{0,0,0\},\qquad  \sigma(e_7)=\{1,1,1\},\\
\sigma(e_i)&=\{0,\alpha,1\},\qquad \sigma(\bar e_i)=\{0,1-\alpha,1\},
\end{align*}
where $e_0=0$ and $i=1,2,3$.
\end{proposition}

\begin{remark}
In the exceptional case $\alpha = \half12$,  there  exists an infinite family of idempotents $c_x:=x_1e_1+x_2e_2+x_3e_3$ on the circle
\begin{equation}\label{circle1}
(x_1-\half13)^2+(x_2-\half13)^2+(x_3-\half13)^2=\half23, \quad x_1+x_2+x_3=1.
\end{equation}
Then $e_1,e_2,e_3,e_4,e_5,e_6$  lie on the circle \eqref{circle1}, while $e_0=0$ and the unit $e_7$ lies outside, see the figure below.

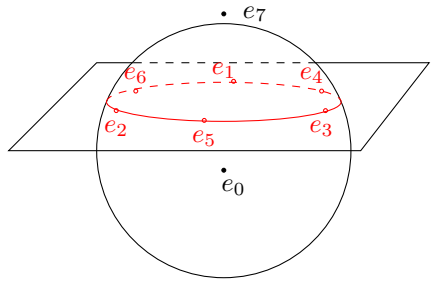
\begin{figure}[ht]
\begin{tikzpicture}[scale=0.13]
\draw (0,0) circle (13cm);
\draw[fill=black] (0,-2) circle (2mm);
\draw (1,-2) node[below] {$e_0$};
\draw  (-9.2,9) -- (-13,9) -- (-22,0) -- (14,0) -- (21,9) -- (9.2,9);
\draw[dashed] (-9.2,9) -- (9.2,9);
\draw[fill=black] (0,14) circle (2mm);
\draw (1,14) node[right] {$e_7$};
\draw[color=red] (-12,5) arc (180:360:12cm and 2cm);
\draw[dashed,color=red] (-12,5) arc (180:0:12cm and 2cm);
\draw[color=red] (-11,4.1) circle (2mm);
\draw[color=red] (-11,4) node[below] {$e_2$};
\draw[color=red] (-2,3.1) circle (2mm);
\draw[color=red] (-2,3) node[below] {$e_5$};
\draw[color=red] (10.4,4.1) circle (2mm);
\draw[color=red] (10,4) node[below] {$e_3$};
\draw[color=red] (10,6.1) circle (2mm);
\draw[color=red] (9,6) node[above] {$e_4$};
\draw[color=red] (1,7.1) circle (2mm);
\draw[color=red] (0,6.5) node[above] {$e_1$};
\draw[color=red] (-9,6.1) circle (2mm);
\draw[color=red] (-9,6) node[above] {$e_6$};

\end{tikzpicture}
\end{figure}

One readily verifies that the Peirce spectrum of all $c$ lying in \eqref{circle1} is the same and is equal to
$\sigma(c)=\{\half12, 1,0\}$. In particular, the only Matsuo algebra $3C_{\frac{1}{2}} $ is a power associative.

\end{remark}

\subsection{The generalized Matsuo algebras}\label{3Cab}

The last example admits the following generalization. Let us define $A=3C_{\alpha,\epsilon}$ being the three-dimensional algebra over $\field$ containing $\alpha$,$\epsilon$,  $\Char \field \ne 2$ and spanned by three idempotents $e_1,e_2,e_3$ subject to the algebra identities
\begin{equation}\label{MatsuoGen}
e_ie_j=\frac{\alpha}{2}(e_i+e_j)+\frac{(\epsilon-\alpha)}{2}e_k, \quad \{i,j,k\}=\{1,2,3\}.
\end{equation}
This obviously determines a unique algebra structure on $\field^3$. Note that
$$
3C_{\alpha}=3C_{\alpha,0},
$$
and the new algebra structure can be thought of as a perturbation of the original structure. Under conditions \eqref{MatsuoGen}, the spectrum of each $e_i$ is found to be
$$
\sigma(e_i)=\{1,\,\,\alpha-\half12\epsilon, \,\, \half12\epsilon \}.
$$
Furthermore, a simple analysis reveals that if
\begin{equation}\label{condit}
(\alpha+1+\epsilon)(\epsilon-1) \gamma\ne 0,
\end{equation}
where
$$
\gamma:=\alpha+1+\epsilon(\epsilon-2\alpha-1),
$$
then there exists exactly $7=2^3-1$ distinct nonzero idempotents in $A$, i.e. the algebra $3C_{\alpha,\epsilon}$ is \textit{generic}.
In that case the remained four idempotents are
\begin{align*}
e_7\,\,\,\,&=\frac{1}{\alpha+1+\epsilon}(e_1+e_2+e_3),
\\
e_{3+i}&=\frac{(1-\epsilon)(\alpha+1+\epsilon)} {\gamma}e_7-
\frac{\alpha+1-2\epsilon}{\gamma}e_i, \quad i=1,2,3.
\end{align*}
One also can readily see that the idempotent $e_7$ is the unity element in $A$ if and only if $\epsilon=0$ (in which case, the algebra $3C_{\alpha,0}$ is isomorphic to the Matsuo algebra $3C_{\alpha}$). The Peirce spectrum then is found to be

$$
\sigma(e_7)=\{1,\mu, \mu\}, \quad \text{where }\mu=1-\frac{3\epsilon}{2(1+\alpha+\epsilon)}
$$
and
$$
\sigma(e_{i+3})=\{1,\,\,\frac{\epsilon(2-\alpha-\epsilon)} {2\gamma}, \,\, \frac{2(1-\alpha^2)+\epsilon(3\alpha-\epsilon-2)} {2\gamma}\},\qquad i=1,2,3.
$$

In summary, for generic $\alpha$ and $\epsilon$ the Peirce spectrum of $A$ except for $0$ and $1$ contains five distinct Peirce numbers\footnote{A  three dimensional generic algebra may a priori have $14=(2^3-1)\times (3-1)$ distinct eigenvalues except $0$ and $1$.}.
It is also straightforward to see that the syzygies, for example in the form \eqref{idemm1}, hold for the obtained  algebra spectrum.

\begin{remark}
In the exceptional cases $\epsilon =1$ and $\epsilon=2\alpha-1$, there  exists an infinite family of (non isolated) idempotents and for $\gamma= 0$ there are 2-nilpotent elements.
\end{remark}

%
%
%
%
%

\section{Metrised algebras}\label{sec:examples}
Note that many  examples discussed in Introduction are \textit{metrized} algebras: they obey a non-trivial  bilinear form $\scal{x}{y}$ which associates with multiplication, i.e.
\begin{equation}\label{Qass}
\scal{xy}{z}=\scal{x}{yz}\quad \text{ for all }x,y,z\in A,
\end{equation}
(cf. \cite{Bordemann}, \cite[p.~453]{Knus}). The standard examples are Lie algebras with the Killing form and Jordan algebras with the generic trace bilinear form.  We say that $A$ is a \textit{Euclidean} metrised algebra if the bilinear form $$\scal{x}{y}$$ is  positive definite. Sometimes an associative bilinear form is also called a trace form \cite{Schafer}.

The condition \eqref{Qass} is very strong and implies that the multiplication operator $L_x:y\to xy$ is self-adjoint for any $x\in A$:
\begin{equation}\label{selff}
\scal{L_xy}{z}=\scal{y}{L_xyz}, \qquad \forall x,y,z\in V.
\end{equation}
In particular, all idempotents are automatically semisimple. If additionally $A$ is an algebra over $\R{}$ then $L_x$ has a real spectrum for any $x$.

\subsection{Algebras of cubic forms}

The category of metrized algebras is in a natural correspondence with the category of cubic forms on vector spaces with a distinguished inner product. Since the latter is an analytic object, it has many computational advances and is a useful tool for constructing diverse examples of non-associative algebras. Below, we briefly recall the correspondence, see also \cite{NTVbook}.

Let $(V,\scal{}{})$ be an inner product vector space, i.e. a vector space over  $\field$ endowed  with a non-singular symmetric $\field$-bilinear form $\scal{x}{y}$. Given a cubic form $u(x)$ on $V$ we define the multiplication by duality:
\begin{equation}\label{muu}
xy :=\text{  the unique element satisfying } \scal{xy}{z}=u(x,y,z) \text{ for all } z\in V
\end{equation}
where
$$
u(x,y,z)=u(x+y+z)-u(x+y)-u(x+z)-u(y+z)+u(x)+u(y)+u(z)
$$
is the full linearization of $u$. Since $\scal{\cdot }{\cdot}$ is nonsingular, such an element $xy$ exists and unique. Thus defined algebra is denoted by $V(u)$ if the definition of $\scal{}{}$ is obvious.

Note also that the trilinear form $u(x,y,z)$ is symmetric, hence
\textit{the algebra $V(u)$ is always commutative.}

It follows from the homogeneity  that the cubic form $u$ is recovered by
\begin{equation}\label{16}
6u(x)=u(x,x,x)=\scal{x}{x^2}.
\end{equation}
Moreover, in this setting, the directional derivative (or the first linearization of $u$) is expressed by
\begin{equation}\label{hesu}
\partial_y u(x)=\half12u(x,x,y)=:u(x;y)
\end{equation}
hence the multiplication is recovered explicitly by
\begin{equation}\label{hessmul}
xy=D^2u(x)y=D^2u(y)x,
\end{equation}
where $D^2u(x)$ is the Hessian matrix of $u$ at $x$. In particular,
the gradient of $u(x)$ is essentially the square of the element $x$ (in $V(u)$):
\begin{equation}\label{grad}
Du(x)=\frac{1}{2}xx=\frac{1}{2}x^2,
\end{equation}

The following correspondence is an immediate corollary of the definitions.

\begin{proposition}\label{pro:cub}
Given a vector space $V$ with a  non-singular symmetric bilinear form $\scal{\cdot}{\cdot}$, there exists a canonical bijection between the vector space of all cubic forms on $V$ and commutative metrised algebras with the multiplication  $(x,y)\to xy$  uniquely determined by \eqref{muu}.
\end{proposition}

\begin{proposition}\label{pro:zeroalgebra}
The metrized algebra $V(u)$ is a zero algebra if and only if $u\equiv 0$.
\end{proposition}

\begin{proof}
Let $A=V(u)$ be a non-zero metrized algebra. Suppose by contradiction that $u(x)\equiv 0$, then $\scal{x}{x^2}\equiv0$ and the polarization yields $\scal{xy}{z}\equiv 0$ for all $x,y,z\in V$, thus $xy=0$ implying $AA=\{0\}$, a contradiction. In the converse direction, if $u(x)\not\equiv 0$ then $u(x_0)=\half16\scal{x_0}{x_0^2}\ne0$ for some $x_0$, which obviously yields $x_0x_0\ne 0$, hence $V(u)$ is a non-zero algebra.
\end{proof}

\begin{proposition}\label{Lem:1}
If $A$ is a nonzero Euclidean metrised algebra then $\Idm(A)\ne \emptyset$.
\end{proposition}

\begin{proof}
Let $S=\{x\in V:\scal{x}{x}=1\}$ for the unit hypersphere $S$ in $V$. Then $S$  is compact in the standard Euclidean topology on $V$.
By \Cref{pro:zeroalgebra}, the cubic form $u(x)=\half16\scal{x^2}{x}\not\equiv0$. Since $u$ is continuous as a function on $S$, it attains its global maximum value there, say at some point $y\in V$, $\scal{y}{y}=1$. Since $u$ is an odd function, the maximum value $u(y)$ is strictly positive. We have the stationary equation $0=\partial_x u|_{y}$ whenever $x\in V$ satisfies the tangential condition $\scal{x}{y}=0$. Thus,  using \eqref{hesu} we obtain
$$
0=\partial_x u|_{y}=3u(y;x)=\half{1}{2}\scal{y^2}{x}
$$
which implies immediately that $y$ and $y^2$ are parallel, i.e. $y^2=ky$, for some $k\in \R{\times}$ (observe also that $k> 0$ by virtue of  $0<u(y)=\scal{y^2}{y}=k\scal{y}{y}$). Scaling $y$ appropriately, namely setting $c=y/k$ yields $c^2=c$.
\end{proof}

\begin{definition}
An idempotent $c\in V(u)$ constructed in the course of proof of \Cref{Lem:1} will be called \textit{extremal}.
\end{definition}

In other words, the set of extremal idempotents coincide with the set  of suitably normalized global maximum  points of the (degree zero homogeneous) function $\scal{x}{x^2}\scal{x}{x}^{-3/2}$. It follows that if $c$ is an extremal idempotent and $c_1$ is an arbitrary idempotent then
$$
\frac{1}{\|c\|}=\frac{\scal{c}{c^2}\,\,\,\,}{\scal{c}{c}^{3/2}}\ge
\frac{\scal{c_1}{c_1^2}\,\,\,\,}{\scal{c_1}{c_1}^{3/2}}=\frac{1}{\|c_1\|},
$$
i.e. the extremal idempotents have the minimal possible norm among all idempotents in $A$.

\begin{remark}
See  \cite{Lyubich1} for a topological proof of the existence of an idempotent element in nonassociative algebras, and also \cite{Lyubich2} for further generalizations.
\end{remark}

\subsection{The exceptional property of $\frac12$ in Euclidean metrised algebras}\label{sec:Except}

\begin{proposition}
\label{pro:max}
Let $A$ be a Euclidean metrised algebra. Then for any extremal idempotent $c$ there holds
$$\sigma(c)\subset (-\infty,\half12].
$$
In particular, $1$ is a simple eigenvalue of $L_c$.
\end{proposition}

\begin{proof}
Then $f(x)=\frac{\scal{x^2}{x}}{|x|^3}$ is a homogeneous of degree zero function which is smooth outside the origin. We have
\begin{equation}\label{D1f}
\partial_y f|_{x}=\frac{3(\scal{x^2}{y}|x|^2-\scal{x^2}{x}\scal{x}{y})}{|x|^5},
\end{equation}
implying
$
\half13Df(x)=\frac{x^2|x|^2-\scal{x^2}{x}x}{|x|^5}.
$
Arguing similarly, we have for the second derivative
$$
\half13\partial_z Df(x)=\frac{2|x|^4xz-3|x|^2(x^2\scal{x}{z}+\scal{x^2}{z}x) -\scal{x^2}{x}|x|^2z +5\scal{x}{z}\scal{x^2}{x}x)}{|x|^7},
$$
hence
\begin{equation}\label{D2f}
\half13D^2f(x)=
\frac{2|x|^4L_x-\scal{x^2}{x}|x|^2-3|x|^2(x\otimes x^2+x^2\otimes x)+5\scal{x^2}{x} x\otimes x}{|x|^5}.
\end{equation}
Now let $c$ be an extremal idempotent of $V$. Then $c^2=c$ and $D^2f(c)\le 0$. The second condition implies
$$
2L_c-1\le c\otimes c.
$$
Since $L_c$ is self-adjoint and $\R{}c$ is an invariant subspace of $L_c$: $L_c=1$ on $\R{}c$, the orthogonal complement $c^\bot=\{x\in V: \scal{x}{c}=0\}$ is an invariant subspace too. Indeed, if $\scal{x}{c}=0$ then
$$
\scal{L_cx}{c}=\scal{x}{L_cc}=\scal{x}{c}=0.
$$
Therefore, using $c\otimes c=0$ on $c^\bot=\{x\in V: \scal{x}{c}=0\}$ we obtain $2L_c-1\le 0$ there, which yields the desired conclusion.
\end{proof}

It turns out that the presence of the exceptional value $\half12$ in the spectrum of a metrized algebra, its corresponding Peirce subspace possesses a certain fusion rule. Recall that any idempotent $c$ in the algebra $A=V(u)$ is  semi-simple, hence
$$
V=\bigoplus_{\lambda\in\sigma(c)} A_c(\lambda),\quad A_c(\lambda)=\ker (L_c-\lambda).
$$

\begin{proposition}
\label{pro:max4}
If $c$ is an extremal idempotent and $\half12\in \sigma(c)$ then $\scal{z^2}{z}=0$ for all $z\in A_c(\half12)$. In other words, the following fusion rule holds:
\begin{equation}\label{fus12}
A_c(\half12)A_c(\half12)\subset A_c(\half12)^\bot.
\end{equation}
\end{proposition}

\begin{proof}
Let $z\in A_c(\half12)$. Using \eqref{D1f} and \eqref{D2f} for $x=c$ yields the directional derivatives respectively
$$
\frac{\partial f}{\partial z}\bigl|_{x=c}=\frac{\partial^2 f}{\partial z^2}\bigl|_{x=c}=0.
$$
Therefore, since $c$ is a local maximum point of $f(x) =\frac{\scal{x^2}{x}}{|x|^3}$, we have for the higher derivatives $\frac{\partial^3 f}{\partial z^3}\bigl|_{x=c}=0$ and $\frac{\partial^4 f}{\partial z^4}\bigl|_{x=c}\,\ge 0$. Using \eqref{D2f} we have
$$
\frac13\frac{\partial^2 f}{\partial z^2}=
\frac{2|x|^2\scal{x}{z^2}-6|x|^2\scal{x}{z}\scal{ x^2}{z}+5\scal{x^2}{x} \scal{x}{z}^2-\scal{x^2}{x}|z|^2}{|x|^7}.
$$
Differentiating the latter identity and substituting $x=c$ yields by virtue of $c^2=c$ and $\scal{c}{z}=0$ that
\begin{align*}
\frac16\frac{\partial^3 f}{\partial z^3}\bigl|_{x=c}&=\frac{\scal{z^2}{z}}{|x|^3}=0.
\end{align*}
Then polarization of $\scal{z^2}{z}=0$ in $A_c(\half12)$ yields the desired fusion rule \eqref{fus12}.
\end{proof}

\subsection{Some examples of algebras of cubic forms}

Below, we consider some concrete examples of cubic forms on the Euclidean space $\R{n}$ and the corresponding metrised algebras. Recall that in this case, all idempotents are semi-simple: roots of the characteristic polynomial $\chi_c(t)$ are always real and the corresponding orthogonal decomposition in eigen-subspaces is the Peirce decomposition associated with $c$.

\begin{example}
\label{sec:red}
Let us consider the algebra of cubic form $u_1=\frac16(x_1^3+\ldots+ x_n^3)$. The algebra multiplication is determined by virtue of \eqref{hessmul}:
$$
xy=(x_1y_1,\ldots,x_ny_n),
$$
i.e. the algebra $V(u_1)$ is reducible and isomorphic to the product $\R{}\times \ldots \times \R{}$. It is easy to see that $V(u_1)$ has exactly  $2^n$ idempotents which coincide with the vertices of the unit cube in $\R{n}$:
$$
c_x=(x_1,\ldots,x_n),\quad \text{where }x_i=0 \text{ or } 1,
$$
The spectrum consists of two different values:
$
\lambda\in\{0, 1\}.
$
For each idempotent,
$$
\sigma(c_x)=\{1^m, 0^{n-m}\}, \quad \text{where }m=x_1+\ldots+x_n.
$$

\end{example}

\begin{example}
Let us consider a perturbation of the algebra $V(u_1)$ of the cubic form from \Cref{sec:red} for $n=3$. Then $V(u_1)\cong \R{}\times \R{}\times\R{}$ with the total spectrum (counting total multiplicities over all $8=2^3$ idempotents in $V(u_1)$)
$$
\biggl(\underbrace{0}_{12}, \, \underbrace{1}_{12}\biggr)
$$
Define  a  perturbation by adding the term $=\epsilon  x_1x_2x_3$, $\epsilon\in \R{}$:
$$
u_{1,\epsilon }(x)=\epsilon  x_1x_2x_3+\half16(x_1^3+x_2^3+x_3^3).
$$
Then one can show that for $\epsilon\not\in\{\pm\half 12, \half14,1\}$ the new algebra is also generic and the corresponding total spectrum is
$$
\biggl(\underbrace{0}_{3}, \, \underbrace{\epsilon}_{3},\,\underbrace{\epsilon}_{ 3},\,  \underbrace{\frac{2\epsilon (1-\epsilon )}{1-2\epsilon +4\epsilon ^2}}_3,\,
\underbrace{1}_{7},\,\underbrace{\frac{1-\epsilon }{1+2\epsilon }}_{2},\,
\underbrace{\frac{1-2\epsilon -2\epsilon ^2}{1-2\epsilon +4\epsilon ^2}}_3\biggr)
$$

Another interesting case when $V(u_{1,\epsilon})$ is still generic, is when $\epsilon=1$. Then  the total spectrum becomes much smaller:
$$
\biggl(\underbrace{0}_{8}, \, \underbrace{-1}_{6},\,\underbrace{1}_{10}\biggr)
$$
In this case the algebra contains the maximal number (totally 8) idempotents with resp. spectrum
\begin{align*}
(0^3)\oplus (0^2,1)\oplus 3\cdot (0,-1,1)\oplus 3\cdot (-1,1,1)
\end{align*}
\end{example}

\begin{example}
Let  $u_2=\frac12x_1(x_3^2-x_4^2)+i x_2x_3x_4$. \label{sec:Hscub}
The algebra $V(u_2)$  over $\mathbb{C}$ has 9 (of $16=2^4$ maximally possible) nonzero idempotents, one (normalized) 2-nilpotents that lying in the two dimensionalsubspace $x_3=x_4=0$. \textit{All idempotents have the same Peirce spectrum}:
$$
-\frac14-\frac{\sqrt{7}}{4}, \,\, -\frac12,\,\, -\frac14+\frac{\sqrt{7}}{4},\,\,1
$$
Note also that the trace $\trace L_c=0$.
\end{example}

\begin{example}\label{ex:same}
Let us consider the algebra $V(u)$ of the cubic form $$u(x)=\frac{1}{6}(3x_1^2+3x_2^2-(4k^2-2)x_3^2)x_3,
\qquad x\in \R{3},
$$
where $k\in\R{\times}$.
Then except for $c_0=(0,0,-\frac{1}{4k^2-2})$, all real nonzero idempotents lie on the circle
$$
c=\{x\in \R{3}:\,x_1^2+x_2^2=k^2, \,\, x_3=\frac12\}.
$$
There are no nontrivial 2-nilpotents in $V(u)$.
The idempotents $c$ lying on the circle have the \textit{same} Peirce spectrum: $\sigma(c)=\{1,\frac12,\frac12-2k^2\}$, while the spectrum $\sigma(c_0)=\{1,-\frac{1}{4k^2-2},-\frac{1}{4k^2-2}\}$.
\end{example}

\bibliographystyle{plain}

\def\cprime{$'$} \def\cprime{$'$}

\end{document}